\documentclass{amsart}[12pt]
\usepackage{amsmath, amsthm, amscd, amsfonts,}
\usepackage{graphicx,float}

\usepackage{amssymb}

\usepackage{pb-diagram}

\newtheorem{theorem}{Theorem}[section]
\newtheorem{lemma}[theorem]{Lemma}

\newtheorem{question}[theorem]{Question}
\newtheorem{definition}[theorem]{Definition}
\newtheorem{example}[theorem]{Example}

\newtheorem{remark}[theorem]{Remark}

\numberwithin{equation}{section}

\def\rmark{\mbox{$\rm\bf\rule{0.06em}{1.45ex}\kern-0.05em R$}}
\def\pmark{\mbox{$\rm\bf\rule{0.06em}{1.45ex}\kern-0.05em P$}}
\def\nmark{\mbox{$\rm\bf\rule{0.06em}{1.45ex}\kern-0.05em N$}}
\def\vdash{\mbox{$\rm\| \kern-0.13em -$}}

\usepackage{pb-diagram}

\def\rmark{\mbox{$\rm\bf\rule{0.06em}{1.45ex}\kern-0.05em R$}}
\def\pmark{\mbox{$\rm\bf\rule{0.06em}{1.45ex}\kern-0.05em P$}}
\def\nmark{\mbox{$\rm\bf\rule{0.06em}{1.45ex}\kern-0.05em N$}}
\def\vdash{\mbox{$\rm\| \kern-0.13em -$}}

\begin{document}

\title[Almost Souslin Kurepa trees]{Almost Souslin Kurepa trees }

\author[M. Golshani]{Mohammad Golshani}

\thanks{The author would like to thank the School of Mathematics,
Institute for Research in Fundamental Sciences (IPM) for their
supports during the preparation of this paper. He also wishes to
thank Dr. E. Eslami and Dr. Sh. Mohsenipour for their inspiration and encouragement.} \maketitle




\begin{abstract}
We show that the existence of an almost Souslin Kurepa tree is
consistent with $ZFC$. We also prove their existence in $L$. These
results answer two questions from [16].
\end{abstract}

\maketitle


\section{Introduction}

The theory of trees forms a significant and highly interesting
part of set theory. In this paper we study $\omega_1-$trees and
prove some consistency results concerning them. Let $T$ be a
normal $\omega_1-$ tree. Let's recall that:

\begin{itemize}
\item $T$ is an Aronszajn tree if it has no branches,  \item $T$
is a Kurepa tree if it has at least $\omega_2-$many branches,
\item $T$ is a Souslin tree if it has no uncountable antichains
(and hence no branches), \item $T$ is an almost Souslin tree if
for any antichain $X \subseteq T,$ the set $S_X = \{ht(x):x \in X
\}$ is not stationary (see [1], [16]), \item $T$ is regressive if
for any limit ordinal $\alpha < \omega_1,$ there is a function
$f:T_{\alpha} \rightarrow T_{< \alpha}$ such that for any $x \in
T_{\alpha}, f(x)<_T x,$ and for any $x \neq y$ in $T_{\alpha},$ at
least one of $f(x)$ or $f(y)$ is above the meet of $x$ and $y$
(see [8]).
\end{itemize}

Intuitively a Kurepa tree is very thick. On the other hand a
Souslin tree is very thin, and obviously no Kurepa tree is a Souslin
tree. We can think of an almost Souslin tree as a fairly thin tree.
The following are well known:
\begin{itemize}
\item There is an Aronszajn tree (Aronszajn, see [3] for proof),  \item It is consistent with $ZFC$ that a Souslin
tree exists (Jech [2], Tennenbaum [13]),
\item $V=L$ implies the existence of a Souslin tree
(Jensen [5], see also [6]),
\item It is consistent with $ZFC+ \neg CH$ to
assume there is no Souslin tree (Solovay and Tennenbaum [11]),
\item It is consistent with $ZFC+GCH$ to assume
there is no Souslin tree (Jensen [7], see [9] for  proof),
\item It is consistent with $ZFC$ that a Kurepa
tree exists (Stewart [12], see [3] for  proof),
\item $V=L$ implies the existence of a Kurepa tree
(Solovay, see [3] for  proof),
\item It is consistent, relative to the existence
of an inaccessible cardinal, that there is no Kurepa tree (Silver [10]).
\end{itemize}

For more details on trees, we refer the readers to the articles
[3] and [14]. The following example shows that almost Souslin
trees exist in $ZFC.$

\begin{example}
 Let $T=\{t \in$$^{< \omega_1}2: Supp(t)$ is finite$ \}$. Then it is easily seen that $T$ is an almost Souslin tree with
$\omega_1-$many branches and with no Aronszajn subtrees (See also
[14] Theorem 4.1).
\end{example}
Now, in [16], Zakrzewski asked the following questions:

\begin{question}
Is the existence of an almost Souslin Kurepa tree
consistent with $ZFC$?
\end{question}

\begin{question}
Does the axiom of constructibility guarantee the
existence of an almost Souslin Kurepa tree?
\end{question}
In this paper, we answer both of these questions positively. In
section 2 we answer question 1.2 by building a model of $ZFC$ in
which an almost Souslin Kurepa tree exists, and in section 3 we
answer question 1.3 by showing that the existence of an
$(\omega_1, 1)-$morass implies the existence of an almost Souslin
Kurepa tree.

\section{Forcing an almost Souslin Kurepa tree}

In this section we answer question 1.2. In fact we will prove
something stronger, which also extends some results from [8].

\begin{theorem}
Assume $GCH$. Then there exists a cardinal preserving
generic extension of $V$ in which an almost Souslin
regressive Kurepa tree exists.
\end{theorem}
\begin{proof}
Let $\kappa \geq \omega_2.$ We produce a cardinal preserving
generic extension of $V$  which contains an almost Souslin
regressive Kurepa tree with $\kappa-$many branches. First we
define a forcing notion $\mathbb{P}$ which adds a regressive
Kurepa tree with $\kappa-$many branches. This forcing is
essentially the forcing notion of [8]. Conditions in $\mathbb{P}$
are of the form $p=\langle T_{p}, \leq_{p}, g_{p}, f_{p}\rangle$
where:

\begin{enumerate}
\item $T_p \subseteq \omega_1$ is countable,
\item $\langle T_p, \leq_p \rangle$ is a normal
$\alpha_p+1-$tree, where $\alpha_p$ is an ordinal less than
$\omega_1,$
\item $g_{p}$ is a bijection from a subset of $\kappa$
onto $(T_p)_{\alpha_p},$ the $\alpha_p-$th level of $T_p$,
\item $f_p: T_{p, lim} \rightarrow T_p,$ where $T_{p, lim}= \{x
\in T_p: ht(x)$ is a limit ordinal $ \},$
\item For all $x \in T_{p, lim}, f_p(x) <_p x,$
\item For each $x \neq y$ in $T_{p, lim},$ if $ht(x)=ht(y),$ then
at least one of $f_p(x)$ or $f_p(y)$ is above the meet of $x$ and
$y$,
\end{enumerate}

The order relation on $\mathbb{P}$ is defined by $p \leq q$ ($p$
is an extension of $q$) iff:

\begin{enumerate}
\item $\langle T_p, \leq_p  \rangle$ end extends  $\langle T_q,
\leq_q  \rangle$,
\item $domg_{p} \supseteq domg_{q},$
\item for all $\alpha \in domg_{q}, g_{p}(\alpha) \geq_{p} g_{q}(\alpha),$
\item $f_p \supseteq f_q.$
\end{enumerate}

The following lemma can be proved easily (see also [8] Theorem 5).
\begin{lemma}
$(a)$ Let $p \in \mathbb{P}$ and $\alpha <
\kappa.$ Then there exists $q \leq p$ such that $\alpha \in
domg_{q}.$ Furthermore $q$ can be chosen so that
$\alpha_{q} = \alpha_{p}+1,$ and $domg_{q}= domg_{p} \cup \{ \alpha \}.$

$(b)$ Let $\langle p_n:n < \omega \rangle$ be a descending
sequence of conditions in $\mathbb{P}.$ Then there exists $q \in
\mathbb{P}$ which extends all of the $p_n$'s. Furthermore $q$ can
be chosen so that $\alpha_q = sup_{n \in \omega}\alpha_{p_n},$ and
$dom(g_{q})= \bigcup_{n < \omega}
dom(g_{p_n}).$

$(c)$ $\mathbb{P}$ satisfies the $\omega_2-c.c.$.\hfill$\Box$
\end{lemma}

It follows from the above lemma that $\mathbb{P}$ is a cardinal
preserving forcing notion. Let $G$ be $\mathbb{P}-$generic over
$V$. Let
\begin{itemize}
\item $T = \bigcup_{p \in G}T_p,$
\item $\leq_T = \bigcup_{p \in G}\leq_p,$
\item $f= \bigcup_{p \in G}f_p: T_{lim}
\rightarrow T,$ where $T_{lim}= \{x \in T: ht(x)$ is a limit
ordinal$ \}.$
\end{itemize}

It is easy to show that $\langle T, \leq_T \rangle$ is a normal
regressive $\omega_1-$tree.
\begin{lemma}
$\langle T, \leq_T \rangle$ has $\kappa-$many branches, in
particular it is a Kurepa tree.
\end{lemma}
\begin{proof}
The lemma follows easily from the following facts:

(1) For each $\xi < \kappa, \{g_{p}(\xi): p \in G, \xi \in
dom(g_{p}) \}$ determines a branch $b_{\xi}$ of $T$.

(2) For $\xi \neq \zeta$ in $\kappa, b_{\xi} \neq
b_{\zeta}.$
\end{proof}

Let $S= \{\alpha_{p}: p \in G, domg_{p}=\bigcup \{domg_{q}: q \in G, \alpha_{q}< \alpha_{p} \}  \}$.
\begin{lemma}
$S$ is a stationary subset of $\omega_1.$
\end{lemma}
\begin{proof}
 Let $\dot{S}$ be a $\mathbb{P}-$name for $S$. Let $p \in \mathbb{P}$ and $\dot{C}$ be a $\mathbb{P}-$name such that

\begin{center}
$p \vdash \ulcorner \dot{C}$ is a club subset of $\omega_1  \urcorner$.
\end{center}

We find $q \leq p$ which forces $\dot{S} \cap \dot{C} \neq
\emptyset.$ Define by induction two sequences $\langle p_n: n <
\omega  \rangle $ of conditions in $\mathbb{P}$ and $\langle
\beta_n: n< \omega \rangle$ of countable ordinals such that:

\begin{itemize}
\item $p_{0}=p,$
\item $p_{n+1} \leq p_{n},$
\item $\alpha_{p_{n}} < \beta_{n} < \alpha_{p_{n+1}},$
\item $p_{n+1} \vdash \ulcorner \beta_{n} \in \dot{C} \urcorner$
\end{itemize}

By Lemma 2.2(b), there is $q \in \mathbb{P}$ such that $q$ extends
$p_{n}$'s, $n < \omega,$ and such that $\alpha_q = sup_{n \in
\omega}\alpha_{p_n},$ and $dom(g_{q})= \bigcup_{n < \omega}
dom(g_{p_n}).$ Then it is easily seen that $q \vdash \ulcorner
\alpha_{q} \in \dot{S} \cap \dot{C}  \urcorner.$
\end{proof}

Working in $V[G],$ let $\mathbb{Q}$ be the usual forcing notion
for adding a club subset to $S$. Thus conditions in $\mathbb{Q}$
are closed bounded subsets of $S$ ordered by end extension. Let
$H$ be $\mathbb{Q}-$generic over $V[G].$  The following is
well-known (see [4] Theorem 23.8).
\begin{lemma}
$(a)$ $\mathbb{Q}$ is $\omega_1-$distributive,

$(b)$ $\mathbb{Q}$ satisfies the $\omega_2-c.c.$,

$(c)$ $C =\bigcup H \subseteq S$ is a club subset of
$\omega_1.$\hfill$\Box$
\end{lemma}

It follows that $\mathbb{Q}$ is a cardinal preserving forcing notion and hence $\langle T, \leq_{T} \rangle$ remains a regressive Kurepa tree with $\kappa-$many branches in $V[G][H]$.  We show that in $V[G][H],  \langle T, \leq_{T} \rangle$ is also almost Souslin.

\begin{lemma}
In $V[G][H],$ $\langle T, \leq_T \rangle$ is almost Souslin.
\end{lemma}
\begin{proof}
 Suppose not. Let $Z \subseteq T$ be an antichain of
$T$ such that $S_Z = \{ht(x): x \in Z \}$ is stationary in
$\omega_1.$ We may further suppose that for $x \neq y$ in $Z,
ht(x) \neq ht(y)$, and that $S_Z \subseteq C.$

First we define a map $h$ on $T \upharpoonright C= \{x \in T: ht(x) \in C \}$ as follows: Let $\alpha \in C$ and $x \in T_{\alpha}.$ Pick $p \in G$ such that $\alpha= \alpha_{p}.$  Then $x= g_{p}(\xi),$ for some $\xi \in domg_{p}$. Let $h(x)=g_{q}(\xi),$ where $q \in G$ is such that $\alpha_{q}$ is minimal with $\xi \in domg_{q}.$ Note that $\alpha_{q} < \alpha_{p},$ and $h(x) <_{T} x$ (as $C \subseteq S$).

The map  $ht(x) \mapsto ht(h(x))$ is well-defined and regressive on $S_Z,$ hence by Fodor's lemma there is $Y \subseteq Z$ and an ordinal $\gamma < \omega_1$ such that $S_Y$ is stationary and for all $x \in Y,
ht(h(x))= \gamma.$ Since $T_{\gamma}$ is countable, we can find $X
\subseteq Y,$ and $t \in T$ such that $S_X$ is stationary, and for
all $x \in X, h(x)= t.$ Now
\begin{center}
$\forall x \in X, \exists p_x \in G, \exists \xi_x \in dom(g_{p_x})(x = g_{p_x}(\xi_x)).$
\end{center}
and then for all $x \in X, h(x)= g_{q_x}(\xi_x),$ where $q_x \in
G$ is such that $\alpha_{q_{x}}$ is minimal with  $\xi_x \in
domg_{q_x}.$ The map $\alpha_{p_{x}} \mapsto \alpha_{q_{x}}$ is
regressive on $S_X$, and hence we can find $W \subseteq X,$ and
$\eta < \omega_1$ such that $S_W$ is stationary and for all $x \in
W, \alpha_{q_{x}}= \eta.$ Let $q \in G$ be such that $\alpha_{q}=
\eta$ Then for all $x \in W, h(x)= g_{q}(\xi_x).$ As $domg_{q}$ is
countable, there are $V \subseteq W$ and $\xi \in domg_{q}$ such
that $S_V$ is stationary and for all $x \in V, \xi_x= \xi.$ Then
for all $x \in V, x = g_{p_x}(\xi).$ Choose $x \neq y$ in $V$, and
let $p \in G$ be such that $p \leq p_x, p_y$. Then $g_{p}(\xi)
\geq_{p} x=g_{p_x}(\xi),y=g_{p_y}(\xi)$. It follows that $x$ and
$y$ are compatible and we get a contradiction. The lemma
follows.
\end{proof}

Thus in $V[G][H],$ $\langle T, \leq_T \rangle$ is an almost
Souslin regressive Kurepa tree with $\kappa-$many branches. This
completes the proof of Theorem 2.1.
\end{proof}
\begin{remark}
As we will see in the next section, just working in $V[G]$, it is possible to define a subtree $T^{*}$ of $T$ which is an almost Souslin regressive Kurepa tree with $\kappa-$many branches. We gave the above argument for Theorem 2.1, since it was our original motivation for defining $T^{*}$.
\end{remark}

\section{Almost Souslin Kurepa trees in $L$}

In this section, answering question 2.2, we show that an almost
Souslin Kurepa tree exists in $L$. Again as in section 2, we prove
something stronger.
\begin{theorem}
If there exists an $(\omega_1, 1)-$morass, then
there is an almost Souslin regressive Kurepa tree.
\end{theorem}
To prove the above theorem, we need some definitions and facts
from [15]. Let $(\mathbb{P}, \leq)$ be a partial order and
$\mathbb{D}= \{ D_{\alpha}: \alpha < \omega_2 \}$ be a family of
open dense subsets of $\mathbb{P}.$ For $p \in \mathbb{P},$ let
$rlm(p)= \{ \alpha < \omega_2: p \in D_{\alpha} \},$ and for
$\alpha < \omega_2$ let $\mathbb{P}_{\alpha} = \{p \in \mathbb{P}:
rlm(p) \subseteq \alpha \}.$ Also let $\mathbb{P}^* =
\bigcup_{\alpha < \omega_1}\mathbb{P}_{\alpha}$.

\begin{definition}
$\mathbb{D}$ is
an $\omega_1-$indiscernible family if the following conditions are
satisfied:

(1) $\mathbb{P}^* \neq \emptyset,$ and for all $\alpha < \omega_1,
\mathbb{P}^* \cap \textit{D}_{\alpha}$ is open dense in
$\mathbb{P}^*$,

(2) For all $\alpha < \omega_1, (\mathbb{P}_{\alpha}, \leq)$ is
$\omega_1-$closed,

Also for each order preserving function $f: \alpha \rightarrow
\gamma, \alpha < \omega_1, \gamma < \omega_2$ there is a function
$\sigma_f: \mathbb{P}_{\alpha} \rightarrow \mathbb{P}_{\gamma}$
such that

(3) $\sigma_f$ is order preserving,

(4) For all $p \in \mathbb{P}_{\alpha},
rlm(\sigma_f(p))=f[rlm(p)]$,

(5) If $\beta < \omega_1, f \upharpoonright \beta = id \upharpoonright \beta, f(\beta)
\geq \alpha, \gamma < \omega_1$ and $p \in \mathbb{P}_{\alpha},$
then $p$ and $\sigma_f(p)$ are

$\hspace{.5cm}$ compatible in $\mathbb{P}^*,$

(6) If $f_1: \alpha_1 \rightarrow \alpha_2, f_2: \alpha_2
\rightarrow \gamma $ are order preserving, $\alpha_1, \alpha_2 <
\omega_1, \gamma < \omega_2,$ then

$\hspace{.5cm}$ $\sigma_{f_2 \circ
f_1}=\sigma_{f_2}\circ \sigma_{f_1}.$
\end{definition}
We also need the following theorem (See [15] Theorem 1.1.3).

\begin{theorem}
The following are equivalent:

$(a)$ There exists an $(\omega_1, 1)-$morass,

$(b)$ Whenever $\mathbb{P}$ is a partial order and $\mathbb{D}$ is
an $\omega_1-$indiscernible family of open dense subsets of
$\mathbb{P}$, then there is a set $G$ which is
$\mathbb{P}-$generic over $\mathbb{D}.$ Furthermore $G$ can be
chosen to be $\omega_1-$complete.\hfill$\Box$
\end{theorem}

We are now ready to give the proof of Theorem 3.1.

Assume an $(\omega_1, 1)-$morass exists. Let $(\mathbb{P}, \leq)$
be the forcing notion of section 2, when $\kappa = \omega_2,$ for
adding a regressive Kurepa tree with $\omega_2-$many branches. For
each $\alpha < \omega_2$ let $\textit{D}_{\alpha}= \{p \in
\mathbb{P}: \alpha \in domg_{p} \},$ and let $\mathbb{D}= \{
D_{\alpha}: \alpha < \omega_2 \}$. Then it is easy to see that for
each $p \in \mathbb{P}, rlm(p)= domg_{p},$ for each $\alpha <
\omega_2, \mathbb{P}_{\alpha}= \{p \in \mathbb{P}: domg_{p}
\subseteq \alpha \},$ and  $\mathbb{P}^*=\mathbb{P}_{\omega_1}.$

By Theorem 1.2.1 of [15],   $\mathbb{D}$ is an
$\omega_1-$indiscernible family of open dense subsets of
$\mathbb{P}.$ Thus using Theorem 3.3, there exists $G \subseteq
\mathbb{P}$ which is $\mathbb{P}-$generic over $\mathbb{D}$ and is
$\omega_1-$complete. Define $T, \leq_T$ and $f$ exactly as in the
proof of Theorem 2.1. By Theorem 1.2.2 of [15], $\langle T, \leq_T
\rangle$ is a normal Kurepa tree, and using $f$ it is regressive.

We now define a subtree of $T$ which is an almost Souslin regressive Kurepa tree. Let $S= \{\alpha_p: p \in G \}$ and let $C$ be the set of limit points of $S$. Then $C$ is a club subset of $\omega_1.$ We first define by induction on $\alpha < \omega_1$ a sequence $\langle T^{\alpha}: \alpha < \omega_1 \rangle$ of subtrees of $T$ as follows:

$\hspace{.5cm}\bullet$ $\alpha=0:$ Let $T^{0}=T$,

$\hspace{.5cm}\bullet$ $\alpha= \beta+1:$ Let $T^{\alpha}=T^{\beta}$,

$\hspace{.5cm}\bullet$ $\alpha$ is a limit ordinal, $\alpha \notin C:$ Let $T^{\alpha}= \bigcap_{\beta < \alpha}T^{\beta}$,

$\hspace{.5cm}\bullet$ $\alpha \in C:$ First let $T^{*}= \bigcap_{\beta < \alpha}T^{\beta}.$ Now we define $T^{\alpha}$ as follows:

$\hspace{.5cm}-$ $(T^{\alpha})_{< \alpha}= (T^{*})_{< \alpha}$,

$\hspace{.5cm}-$ $(T^{\alpha})_{\alpha}= \{x \in (T^{*})_{\alpha}: \exists \xi < \kappa, \forall p \in G (\alpha_{p} < \alpha \wedge \xi \in domg_p \Rightarrow g_{p}(\xi) \leq_{T} x)  \}$,

$\hspace{.5cm}-$ for $\gamma > \alpha, (T^{\alpha})_{\gamma}= \{x \in (T^{*})_{\gamma}: \exists y \in (T^{\alpha})_{\alpha}, x \geq_{T} y \}$.

\begin{remark}
For $x \in (T^{\alpha})_{\alpha},$ the required $\xi$ is unique. Furthermore If $y \in T^{\alpha}$, and $x \leq_{T} y,$ then $x \in T^{\alpha}.$
\end{remark}

Finally let $T^{*}= \bigcap_{\alpha < \omega_1}T^{\alpha}.$ Clearly $T^{*}$ is a subtree of $T$ with $\omega_2-$many branches, and hence it is a regressive Kurepa tree. We show that it is almost Souslin.

For $\alpha \in C$ we define $g_{\alpha}: \bigcup \{domg_{p}: p \in G, \alpha_{p} < \alpha \} \rightarrow (T^{*})_{\alpha}$ as follows: Let $\xi \in domg_{p},$ where $p \in G, \alpha_{p} < \alpha $. Then there is a unique $x \in (T^{*})_{\alpha}$ such that:

\begin{center}
$\forall q \in G (\alpha_{q} < \alpha \wedge \xi \in domg_{q} \Rightarrow x \geq_{T} g_{q}(\xi)).$
\end{center}

Let $g_{\alpha}(\xi)=x.$ Let us note that $g_{\alpha}$ is a
bijection and for $\alpha \in C \cap S, g_{\alpha}=g_{p}$ where $p
\in G$ is such that $\alpha_{p}=\alpha.$ Next we define a function
$h: T^{*} \upharpoonright C \rightarrow T^{*}$ as follows: Let $\alpha \in C$
and $x \in (T^{*})_{\alpha}.$ Then $x=g_{\alpha}(\xi)$ for some
$\xi \in domg_{\alpha}.$ Let $h(x)=g_{\beta}(\xi),$ where $\beta$
is the least ordinal such that for some $p \in G,
\beta=\alpha_{p},$ and $\xi \in domg_{p}.$ Note that $\beta <
\alpha$ and $h(x)<_{T} x.$ Now as in the proof of Lemma 2.6, we
can show that $T^{*}$ is almost Souslin. Hence $T^{*}$ is an
almost Souslin regressive Kurepa tree. This completes the proof of
Theorem 3.1.\hfill$\Box$

\begin{remark}
The methods of this paper can be used to get more consistency results about trees. For example we can show that the existence of an almost Souslin Kurepa tree with no Aronszajn subtrees is consistent with $ZFC,$ and that such a tree exists in $L$.
\end{remark}
The following question from [16] remained open.

\begin{question}
Does there exist a Souslin tree T such that for each
$G$ which is $T-$generic over $V, T$ is an almost Souslin
Kurepa tree in $V[G]$?
\end{question}

\bibliographystyle{amsplain}

Department of Mathematics, Shahid Bahonar University of Kerman,
Kerman-Iran and School of Mathematics, Institute for Research in
Fundamental Sciences (IPM), Tehran-Iran.

golshani.m@gmail.com

\end{document}